\documentclass[11pt]{article}

\usepackage[margin=1in]{geometry}

\usepackage{amsthm,amsmath,amssymb,mathtools}
\usepackage{enumitem,graphicx,color,url}
\usepackage{authblk,setspace,booktabs}
\usepackage[numbers,sort&compress]{natbib}

\theoremstyle{plain}
\newtheorem{theorem}{Theorem}[section]
\newtheorem{remark}[theorem]{Remark}
\newtheorem{lemma}[theorem]{Lemma}
\newtheorem{proposition}[theorem]{Proposition}
\newtheorem{claim}[theorem]{Claim}
\newtheorem*{theorem*}{Theorem}
\newtheorem*{claim*}{Claim}
\newtheorem*{lemma*}{Lemma}
\newtheorem*{remark*}{Remark}
\newcommand{\R}{\mathbb{R}}
\newcommand{\E}{\mathbb{E}}
\newcommand{\tX}{\tilde{X}}
\newcommand{\tY}{\tilde{Y}}

\newcommand{\Set}{\mathcal{S}}

\newcommand{\e}{\varepsilon}
\newcommand{\la}{\lambda}
\newcommand{\be}{\beta}
\newcommand{\ga}{\gamma}
\newcommand{\te}{\theta}
\newcommand{\hte}{\hat{\theta}}

\newcommand{\mtr}{\mathrm{tr}}

\author[1]{Chao Lan\thanks{Email: clan@ou.edu}}
\author[1]{Luyuan Yang}
\affil[1]{School of Computer Science, 
University of Oklahoma}
\date{}

\begin{document}

\title{Bidirectional Random Projections}

\maketitle

\begin{abstract}
This paper analyzes  
bidirectional random 
projections for ordinary 
least squares (OLS) regression 
under the fixed design setting. 
Let $(X, Y) \in \mathbb{R}^{n\times p} 
\times\mathbb{R}^{n}$ be a sample and
$W \in \mathbb{R}^{n_1\times n}$, $R \in \mathbb{R}^{p\times p_1}$ be two 
random projections. We develop an 
expected excess loss bound for the
OLS estimator built on $(WXR, WY)$, 
extending a prior bound for the 
estimator built on $(XR, Y)$. Our analysis shows 
that our bound is larger and, as $n_1$ increases, 
the gap between these bounds decays at an 
approximately polynomial rate. 
This is consistent with our numerical results. 
\end{abstract}

\section{Introduction}

Random projection is used to 
efficiently scale up ordinary least squares (OLS) regression by reducing data dimension \cite{maillard2009compressed,fard2012compressed,kaban2014new,slawski2017compressed} or sample size \cite{zhou2007compressed,raskutti2016statistical,dobriban2019asymptotics}. 
On massive high-dimensional data, 
it is natural to use it to reduce both, 
an approach we call 
\textit{bidirectional random 
projections}. However, to our knowledge, 
statistical properties of the resulting OLS 
estimator remain unexplored. 

This paper develops an 
excess loss bound for the 
resulting estimator, which extends 
an established bound \cite{kaban2014new} 
for the OLS estimator that uses 
random projection to reduce 
data dimension. Our analysis 
shows that the new bound is 
larger and, as $n_1$ increases, 
the gap between the two bounds 
decays at an approximately 
polynomial rate. This is consistent 
with the numerical results. 

\section{Preliminaries}

Consider the fixed design setting. 
Let $x_1, \ldots, x_n \in \R^p$ be 
$n$ fixed points and $y_1, \ldots, 
y_n$ be their labels generated by 
$y_i = x_i^T \be_* + \e_i$, 
where $\be_* \in R^p$ is fixed 
and $\e_i$'s are independent 
random noises with $\E \e_i = 0$ 
and $var(\e_i) = \sigma^2$ for 
some fixed $\sigma >0$.\ \  
Let $X \in \R^{n \times p}$ 
be a sample matrix with 
$x_i^T$ on the $i_{th}$ row, 
$Y \in \R^n$ be a label 
vector with $Y_i = y_i$ 
and $E \in \R^n$ be a 
noise vector with $E_i = \e_i$. 
Then, the above setting 
can be written in a matrix form 
\begin{equation}
Y = X \be_* + E.    
\end{equation}

Consider two random projection 
matrices. Let $W \in \R^{n_1 \times n}$ 
be the matrix for reducing sample 
size from $n$ to $n_1 (< n)$, and 
$R \in \R^{p \times p_1}$ be 
the matrix for reducing data 
dimension from $p$ to $p_1 (<p)$;
assume $p_1 + 1 \leq n_1$. 
Assume $W$ has i.i.d. 
$N(0, 1/n)$ entries and $R$ has 
i.i.d. $N(0,1/p_1)$ entries. 

Consider two OLS estimators. 
Let $\hat{\te}_R \in R^{p_1}$ 
be the OLS estimator built on 
$(XR, Y)$, i.e., 
\begin{equation}
\label{eq:def_teR}
\hat{\te}_R = \arg\min_{\te} 
\frac{1}{n} \parallel X R \te 
- Y \parallel^2,  
\end{equation}
where $||\cdot||$ is the $L_2$ norm. 
Let $\hat{\te}_{RW}$ be the OLS 
estimator built on $(WXR, WY)$, i.e., 
\begin{equation}
\hat{\te}_{RW} = \arg\min_{\te} 
\frac{1}{n_1} \parallel W X R \te - W Y \parallel^2.
\end{equation}

We are interested in 
developing an excess loss 
bound for $\hat{\te}_{RW}$ 
and comparing it to the 
established bound for $\hat{\te}_R$. 
To this end, define the 
expected squared loss of 
any estimator $\be \in \R^p$ as 
\begin{equation}
L(\be) = \E_{Y}\ \frac{1}{n} 
\parallel X \be - Y \parallel^2. 
\end{equation}
Given $R$, define the expected squared 
loss of any estimator $\te \in \R^{p_1}$ as 
\begin{equation}
L_R(\te) = \E_{Y \mid R}\ \frac{1}{n} 
\parallel X R \te - Y \parallel^2.     
\end{equation}

Let $I_s$ be an identity matrix of 
size $s$; for matrix $G$, let 
$||\cdot||_G$ be the induced 
Mahalanobis norm, $\mtr(G)$ be its 
trace, $\lambda_{\uparrow}(G)$ be 
the largest singular value and 
$\lambda_{\downarrow}(G)$ be the 
smallest singular value. Notation 
$\E_{Z}$ means an expectation 
is w.r.t. the randomness 
in $Z$, and $\E_{Z \mid Q}$ means 
the expectation is conditioned on $Q$. 
An $\E$ without subscript has 
the same subscript as its preceding one. 

\subsection{A Sketching Constant}

We introduce a constant for 
characterizing the new bound. 
Assume 
$\E_{R,W,Y} \|XR\hat\theta_{RW} -Y\|^2 
\in (0,\infty)$.
Let $\Set$ be the column span of 
$[XR, Y]$ and $k_s$ be its dimension. 
Since $[XR, Y]$ has $p_1 + 1$ columns, 
we have $k_s \leq p_1 + 1$. Define constant 
\begin{equation}
\label{eq:lambda0}
\la_0\ := \frac{\E_{R,W,Y} [ 
\la_{\downarrow \Set}^2(W)\, 
\|XR\hat\theta_{RW}-Y\|^2 ]}
{\E_{R,W,Y} \|XR\hat\theta_{RW} -Y\|^2}, 
\end{equation}
where 
\begin{equation}
\label{eq:def_swu}
\lambda_{\downarrow \Set}(W)\ 
:= \min_{z\in\mathcal{S},\, z\neq 0} \frac{\|Wz\|}{\|z\|}.  
\end{equation}

We can view $\lambda_0$ as a measure of 
how much the left projection $W$ can shrink 
the  loss of $\hte_{RW}$ on the right-projected 
sample $(XR,Y)$, with a smaller $\lambda_0$ 
indicating more shrinkage. It serves 
as a bridge for us to extend 
the existing loss bound on the right-projected sample to the bidirectionally projected sample. 
Two immediate properties of $\lambda_0$  
are as follows: 
\begin{itemize}
\item $\la_0 > 0$, 
because $\lambda_{\downarrow 
\Set}(W) > 0$ almost surely 
given $k_s \le p_1+1\le n_1$.
\item $\la_0 = 1$ 
if $W$ is fixed to $I_n$.  
\end{itemize}

A deeper understanding of $\la_0$ 
can be obtained by inspecting $W$ 
through an orthonormal basis $U$ 
for $\Set$. The following lemma 
connects $\sqrt{n} WU$ with $W$. 
The former is easier to analyze 
since it has a positive least 
singular value almost surely, 
while the latter has a zero least singular 
value.\footnote{In this paper, we follow 
a convention that singular values of 
$A \in \mathbb{R}^{m \times n}$ are 
the square roots of the eigenvalues of 
$A^T A$ \cite{lay2003linear}. Thus, 
if $A$ is wide (i.e., $m < n$), $\la_{\downarrow}(A) = 0$ 
which would make our bound vacuous as 
it relies on $1/\la_{\downarrow}(A)$. 
} 

\begin{lemma}
\label{lemma:baseU}
Let $U \in \R^{n \times k_s}$ have 
orthonormal columns that form a basis 
for $\Set$. Then, given $(R,Y)$, 
it follows that (i) $\la_{\downarrow}(WU) 
= \la_{\downarrow \Set}(W)$ 
and (ii) $\sqrt{n} WU$ has 
i.i.d. $N(0,1)$ entries. 
\end{lemma}

Using the lemma, we obtain the following 
two results: the first one follows from a 
concentration bound for Lipschitz functions 
of Gaussian vectors  
\cite[Proposition 5.34]{vershynin2012introduction}
plus the fact that the least singular 
value function is Lipschitz 
\cite[Corollary 7.3.5]{horn2012matrix}; 
the second result follows 
from the classical Gordon's 
theorem for Gaussian matrices 
\cite[Theorem 5.32]{vershynin2012introduction}.

\begin{lemma}
\label{claim:SVW_concentration}
Given $(R,Y)$, for any $t > 0$, 
$|\, \la_{\downarrow \Set}(W) 
- \E_{W \mid R, Y} 
[\la_{\downarrow \Set}(W)]\, 
|  \leq t  / \sqrt{n}$ 
with probability at least 
$1 - 2 e^{-t^2/2}$ over 
the random choice of $W$. 
\end{lemma}

\begin{lemma}
\label{claim:SVW_mean}
Given $(R,Y)$, 
$\frac{\sqrt{n_1} - \sqrt{p_1+1}}{
\sqrt{n}}\ \leq\  \E_{W \mid R, Y} 
[\la_{\downarrow \Set}(W)]\ 
\leq\ \sqrt{\frac{n_1}{n}}$.
\end{lemma}

Informally, Lemma \ref{claim:SVW_concentration} 
states that $\la_{\downarrow \Set}(W)$ 
concentrates around its mean and Lemma 
\ref{claim:SVW_mean} states that the 
mean is of order $\sqrt{n_1/n}$ 
in a certain asymptotic regime. 
Together, they imply that 
$\la_{\downarrow \Set}(W)$ concentrates 
tightly around its mean, based on which 
we can deduce that $\la^2_{\downarrow \Set}(W)$ 
also concentrates tightly around 
$\E^2 \la_{\downarrow \Set}(W)$. 
Plugging this into $\la_0$ 
yields the following. 

\begin{remark}
\label{rem:la0_estimate}
As $n_1 \to \infty$ 
and $p_1/n_1 \to \gamma$ for some 
constant $\gamma \in [0, 1)$, 
\begin{equation}
\label{eq:la0_approx}
\la_0 \approx c\ \in \left[ 
\frac{n_1}{n}\left( 1 - 
\sqrt{\gamma}\right)^2,\ 
\frac{n_1}{n}\right]. 
\end{equation}
In particular, if $p_1$ is fixed, 
then $\la_0 \approx n_1/n$. 
\end{remark}

\section{Main Result}
\label{sec:mainresult} 

We first revisit an established  
loss bound for $\hat{\te}_R$ 
\cite{kaban2014new}. 
The original theorem assumes a more general 
distribution of $R$ and therefore 
carries an extra term involving 
the excess kurtosis. In this paper, 
we specialize it to a centered Gaussian, 
which has zero excess kurtosis.

\begin{theorem}[\cite{kaban2014new}]
\label{thm:rpbound}
For $p_1 < \min(p, n)$: 
\begin{equation}
\E_{R, \hte_R} [L_{R}(\hat{\te}_R)] 
- L(\be_*) \leq \sigma^2 \frac{p_1}{n} 
+ \frac{1}{p_1} || \be_* ||^2_{M}\ 
=: B_R,   
\end{equation}
where $M = \Sigma + \mtr(\Sigma) I_p$ and 
$\Sigma = \frac{1}{n} X^T X$. 
\end{theorem}

\vspace{5pt}
Our bound for $\hat{\te}_{RW}$ is 
stated as follows. 

\begin{theorem}
\label{cor:brpbound}
For $p_1 < \min(p, n_1)$: 
\begin{equation}
\label{eq:brpbound}
\E_{R, W, \hte_{RW}} 
[L_{R}(\hat{\te}_{RW})] - L(\be_*) 
\leq C_1 \sigma^2 \frac{p_1}{n} 
+ C_2 \frac{1}{p_1} || \be_* ||_{M}^2 
 + b\ =: B_{RW},  
\end{equation}
where $C_1 = \E_W[\la^2_{\uparrow}(W)]\, 
\la^{-1}_{0}$, $C_2 = \frac{n_1}{n} \la^{-1}_{0}$ 
and $b = (C_2 - 1) \sigma^2$.
\end{theorem}

\vspace{5pt}
To better interpret the new 
bound, we state two claims. 
The first one follows directly from 
the definitions, and the second one
is proved in the appendix.

\begin{claim}
\label{claim:Ob2a}
If $W = I_n$, then $C_1 = C_2 = 1$ 
and $b = 0$, which implies 
$B_R = B_{RW}$.\footnote{This claim is 
primarily used as a degenerate sanity check.}
\end{claim}

\begin{claim}
\label{claim:Ob2}
If $\E_{R,W,Y}\| 
XR \hte_{RW} - Y \|^2 \in (0, \infty)$, 
then $C_1 \ge 1$ and $C_2 \geq 1$, 
which implies $B_{RW} \geq B_R$.
\end{claim}

\vspace{5pt}
Based on these claims, we compare the 
two bounds and obtain several implications. 

\vspace{10pt}
\noindent \textbf{Implication 1:}
$B_R = B_{RW}$ if $W = I_n$. 

\vspace{5pt}
This implication makes sense, since 
bidirectional projection becomes  
feature projection in that case, 
so their resulting OLS estimators 
become equivalent.

\vspace{10pt}
\noindent \textbf{Implication 2:}
$B_{RW} \geq B_R$ in general. 

\vspace{5pt}
This implication also makes sense, since 
$\hte_{RW}$ introduces an extra data distortion 
given by $W$ and thus should perform 
worse than $\hte_R$. 

Our next question is how fast 
$B_{RW}$ approaches $B_R$ as 
$n_1$ increases. 
This requires some evaluation 
of the constants. For 
$n \gg n_1 \gg p_1$, 
\begin{equation}
\label{eq:C1app}
C_1 = \E_W[\lambda^2_{\uparrow}(W)] 
\cdot \lambda^{-1}_{0} \lesssim 
\left(1 + \sqrt{\frac{n_1}{n}}\right)^2
\cdot \frac{n}{n_1} \left( 1 - 
\sqrt{\gamma}\right)^{-2} 
\approx \frac{n}{n_1} (1 + O(\sqrt{\gamma})).  
\end{equation}
In (\ref{eq:C1app}), the approximate inequality 
is based on Remark \ref{rem:la0_estimate} 
and a classical singular value bound 
for Gaussian matrix \cite[Corollary 5.35]{vershynin2012introduction} (which says 
that $\la_{\uparrow}(\sqrt{n} W) 
\leq \sqrt{n} + \sqrt{n_1}$ with 
high probability); the approximate equality 
follows from $n \gg n_1$ and the first-order 
Taylor expansion (which states that 
$\left( 1 - \sqrt{\gamma}\right)^{-2} 
= 1 + 2 \sqrt{\ga} + O(\ga)$ for small 
enough $\ga > 0$). 

Evaluation of $C_2$ is based on similar arguments:  
\begin{equation}
C_2 = \frac{n_1}{n} \la^{-1}_{0} 
\lesssim \frac{n_1}{n} \frac{n}{n_1} 
\left( 1 - \sqrt{\gamma}\right)^{-2} 
= 1 + O(\sqrt{\ga}).
\end{equation}

Putting it all together and plugging in 
$\ga \approx \frac{p_1}{n_1}$ with 
a fixed $p_1$, 
\begin{equation}
\label{eq:biggapanalysis}
\begin{split}
B_{RW} - B_R 
& = (C_1 -1) \sigma^2 \frac{p_1}{n} 
+ (C_2 - 1) \frac{1}{p_1} 
\| \be_*\|_M^2 + b \\
& = (C_1 -1) \sigma^2 \frac{p_1}{n} 
+ (C_2 - 1) \big(\frac{1}{p_1} 
\| \be_*\|_M^2 + \sigma^2 \big)\\ 
& \lesssim \sigma^2 \frac{p_1}{n_1} 
 (1 + O(\sqrt{\gamma})) 
+ O(\sqrt{\ga}) \big(\frac{1}{p_1} 
\| \be_*\|_M^2 + \sigma^2 \big)\\
& = O\big(\frac{1}{n_1^{1.5}} 
+ \frac{1}{n_1} 
+ \frac{1}{n_1^{0.5}} \big).
\end{split}
\end{equation}
This suggests the following, which is 
consistent with our numerical results. 

\vspace{10pt}
\noindent \textbf{Implication 3:}
$B_{RW}$ approaches $B_R$ polynomially 
fast as $n_1$ increases. 

\vspace{5pt}
Finally, we inspect the approaching 
rate as $p_1$ increases. Analysis 
in (\ref{eq:biggapanalysis}) 
suggests that 
\begin{equation}
\label{eq:bigOp1}
B_{RW} - B_R 
\lesssim O(p_1^{1.5} + p_1 + p_1^{0.5} 
+ p_1^{-0.5}).
\end{equation}
When the $p_1^{-0.5}$ term is large, this 
seems to suggest the gap has a U-shape as 
$p_1$ increases. 
However, only increasing gaps are observed 
in our numerical results. One explanation 
is that our bound is loose (partly due to 
the $\|W z\| \leq \la_{\uparrow}(W) 
\| z\|$ argument in analysis) and thus cannot 
accurately characterize the pattern 
for $p_1$; another possible explanation 
is that the pivot of the U-shape 
(which depends on the hidden constants 
of different terms) is too small to attain 
in the experiment.

\section{Proof Sketch of Theorem \ref{cor:brpbound}}

Given $(R,W)$, let the expected 
squared loss of any estimator $\te$ be  
\begin{equation}
L_{RW}(\te) = \E_{Y \mid R, W}\ 
\frac{1}{n_1} \| W X R \te 
- W Y \|^2.     
\end{equation} 
Let $\hte_*$ be 
a minimizer of $L_{RW}(\te)$, i.e., 
\begin{equation}
\hte_* = \arg\min_{\te} L_{RW}(\te).
\end{equation}
We will focus on bounding 
$\E [L_{RW}(\hte_{RW})]$ 
and then connect it to 
$\E [L_{R}(\hte_{RW})]$. 

Given $(R,W)$, it can be 
shown that (Claim \ref{claim:02}) 
\begin{equation}
\E_{\hte_{RW} \mid R,W} 
[\hte_{RW}] = \hte_*,   
\end{equation}
which implies (Claim \ref{claim:03}) 
\begin{equation}
\label{eq:road_01a}
\E_{\hte_{RW} \mid R,W} [L_{RW}(\hte_{RW})] \leq \E \frac{1}{n_1}
|| W XR \hte_{RW} - W XR \hte_* ||^2 
+ L_{RW}(\hte_*).
\end{equation}
On the right side, the first term 
satisfies (Claim \ref{claim:04}) 
\begin{equation}
\label{eq:road_02}
\E_{\hte_{RW} \mid R,W} 
\frac{1}{n_1} 
|| WXR \hte_{RW} - WXR \hte_* ||^2 
\leq \sigma^2 \frac{p_1}{n_1} 
\la_{\uparrow}^2(W). 
\end{equation}
For the second term, since 
$\hte_*$ minimizes $L_{RW}(\te)$, 
\begin{equation}
L_{RW}(\hte_*) \leq 
L_{RW}(R^T\be_*).    
\end{equation}
Plugging both back to 
(\ref{eq:road_01a}) and 
taking expectation w.r.t. 
$W$ on both sides, 
\begin{equation}
\label{eq:road_04bb}
\begin{split}
\E_{W, \hte_{RW} \mid R} 
[L_{RW}(\hte_{RW})] 
& \leq \sigma^2 \frac{p_1}{n_1} 
\E_{W} [\la_{\uparrow}^2(W)] 
+ \E_{W \mid R} [L_{RW}(R^T\be_*)]\\[.5em]
& = \sigma^2 \frac{p_1}{n_1} 
\E [\la_{\uparrow}^2(W)] 
+ L_{R}(R^T\be_*),       
\end{split}
\end{equation}
where the equality holds by the fact that 
$\E_W ||W z||^2 = \frac{n_1}{n} ||z||^2$ 
for any $z$ independent of $W$ so that
\begin{equation}
\begin{split}
\E_{W \mid R} [L_{RW}(R^T\be_*)]
& = \frac{1}{n_1} \E_{W,Y \mid R} 
\| W X R (R^T\be_*) - W Y \|^2\\
& = \frac{1}{n_1} \frac{n_1}{n} \E_{Y \mid R} 
\| XR (R^T\be_*) - Y\|^2 \\
& = L_R(R^T \be_*).        
\end{split}
\end{equation}
Taking expectation w.r.t. $R$ on both 
sides of (\ref{eq:road_04bb}), 
\begin{equation}
\label{eq:road_08a}
\begin{split}
\E_{W, R, \hte_{RW}} 
[L_{RW}(\hte_{RW})] 
& \leq \sigma^2 \frac{p_1}{n_1} 
\E_W [\la_{\uparrow}^2(W)] 
+ \E_R [L_{R}(R^T\be_*)]\\ 
& \leq \sigma^2 \frac{p_1}{n_1} 
\E_W [\la_{\uparrow}^2(W)] 
+ L(\be_*) + \frac{1}{p_1} 
||\be_*||_M, 
\end{split}
\end{equation}
where $M = \Sigma + tr(\Sigma) I_p$ and 
$\Sigma = \frac{1}{n} X^T X$;  
the second line follows directly 
from a bound in \cite{kaban2014new}
that $\E_{R} [L_{R}(R^T\be_*)] 
\leq L(\be_*) + \frac{1}{p_1} 
||\be_*||_M^2$  (rephrased in 
Remark \ref{rem:priorkaban}). 

Finally, we connect 
$\E [L_{RW}(\hte_{RW})]$ 
to $\E [L_{R}(\hte_{RW})]$
by Claim \ref{claim:05}:  
\begin{equation}
\label{eq:road_09a}
\E_{R,W,\hte_{RW}}[L_{RW}(\hte_{RW})] \;\ge\;\frac{n}{n_1}\,\la_0\;\E_{R,\hte_{RW}}[L_R(\hte_{RW})].
\end{equation}
Combining (\ref{eq:road_08a}, 
\ref{eq:road_09a}) and noting 
that $L(\be_*) = \sigma^2$ (Claim 
\ref{claim:05b}) proves the theorem.

\section{Numerical Results}

We evaluated both estimators on three 
real-world data sets, namely
the Digit data set and MNIST data 
set that are directly loaded from the 
Python sklearn and Torch libraries 
respectively, and the Superconductivity 
(Super) data set 
from\footnote{\url{https://doi.org/10.24432/C53P47}}. 
Constant and linearly dependent features 
were removed for numerical stability. 
Since $Y$ is fixed on real-world 
data, all expectations were estimated 
based on the random sampling of $(R,W)$, 
and $\be_*$ was estimated as the 
minimizer of $\frac{1}{n}||X \be - Y||^2$. 
All results were averaged over 100 
random choices of $(R,W)$. 
Source code of our experiment is 
at\footnote{\url{https://github.com/petelan/Bidirectional-Random-Projection.git}}.

We first fixed $p_1 = p/4$ and 
varied $n_1$. The excess losses of 
both estimators and the corresponding 
loss gap are shown in Figure \ref{fig:error}. 
We see the loss of $\hte_{RW}$ is 
larger than that of $\hte_{R}$, 
which is consistent with Implication 2. 
The gap figures have both axes 
on a logarithmic scale. Since all curves are 
almost linear, they indicate polynomial 
decays of the gap which is consistent 
with Implication 3. Moreover, on Digit 
and Super, the slopes are close to $-1$, 
suggesting that the gap behaves more 
like the $O(1/n_1)$ term in the gap 
bound (\ref{eq:biggapanalysis})  
within the examined range of $n_1$.

We then fixed $n_1 = n/4$ and 
varied $p_1$. 
Results are shown in Figure \ref{fig:error_p}. 
Again, the loss of $\hte_{RW}$ is 
larger than the loss of $\hte_{R}$, 
which is consistent with Implication 2. 
However, no U-shape is observed. 
All gaps scale almost linearly in $p_1$, 
which is consistent with the $O(p_1)$ 
term in (\ref{eq:bigOp1}).

\begin{figure}
\centering
\includegraphics[width=.32\linewidth]{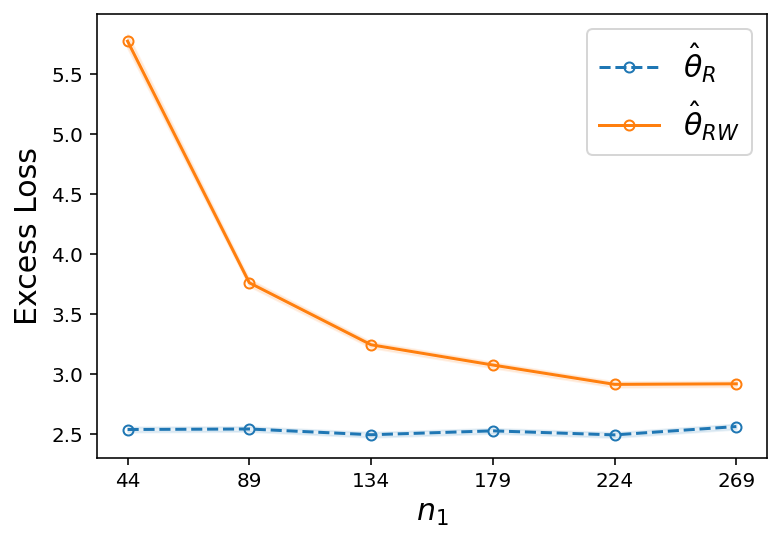}
\includegraphics[width=.32\linewidth]{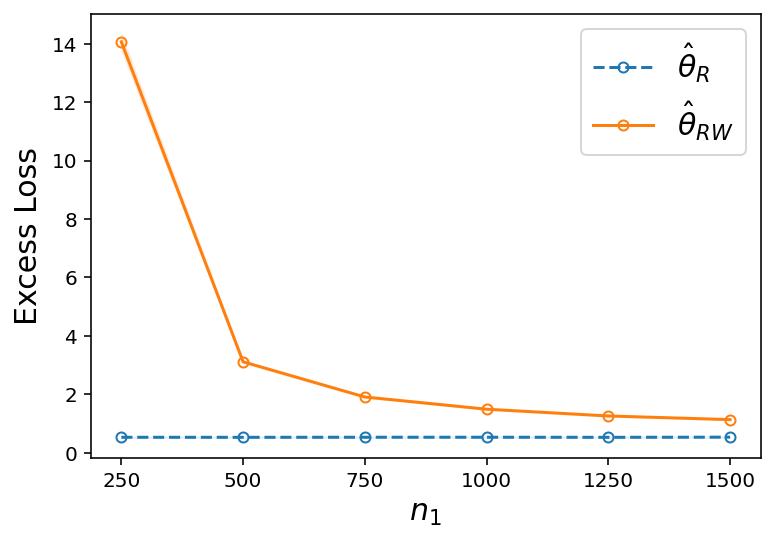}
\includegraphics[width=.32\linewidth]{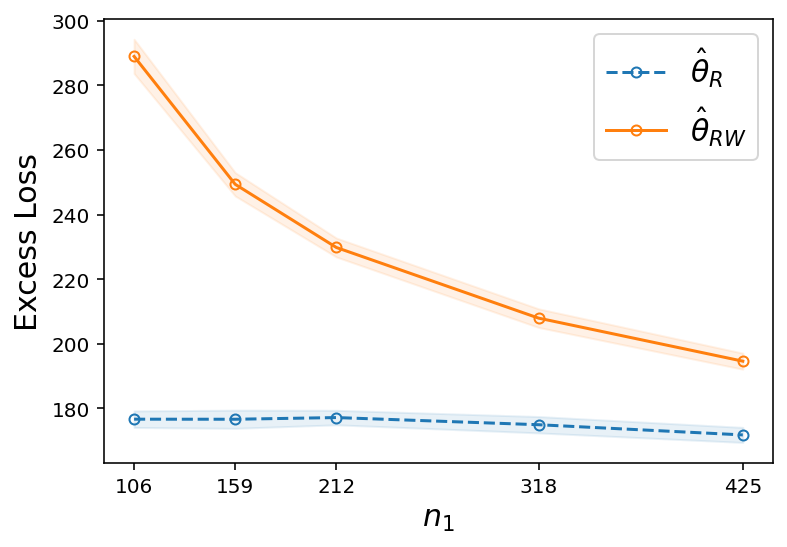}\\
\includegraphics[width=.32\linewidth]{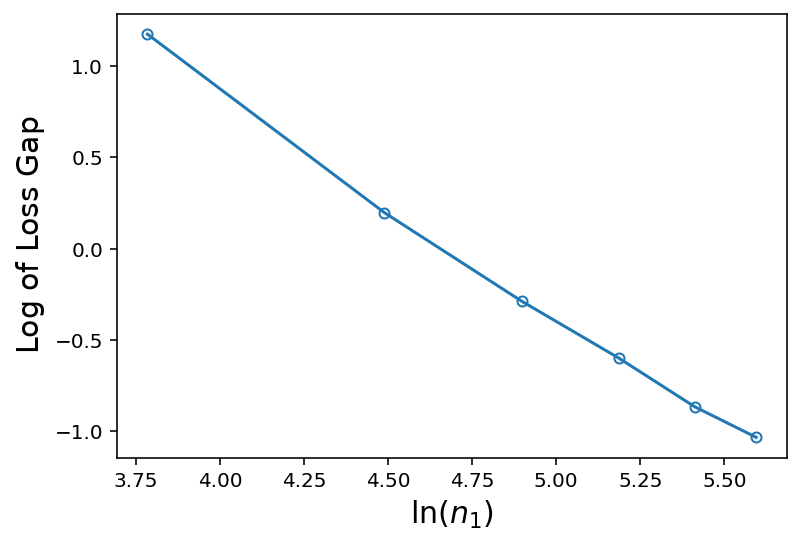}
\includegraphics[width=.32\linewidth]{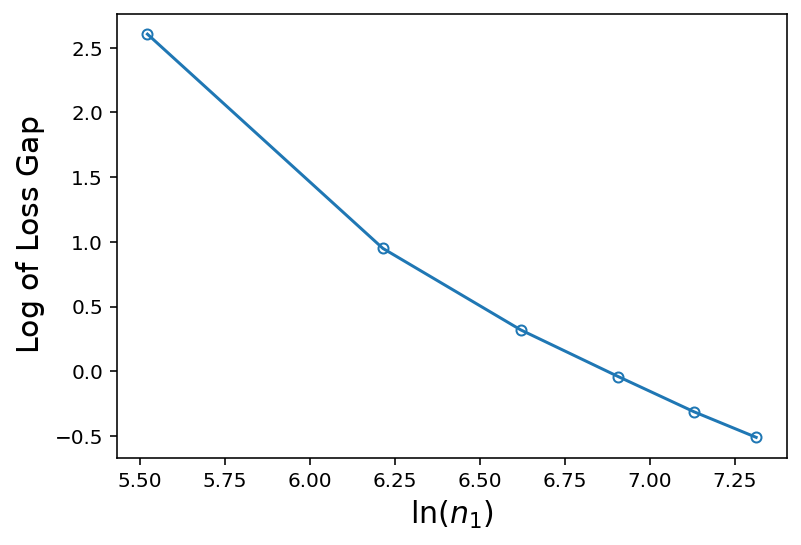}
\includegraphics[width=.32\linewidth]{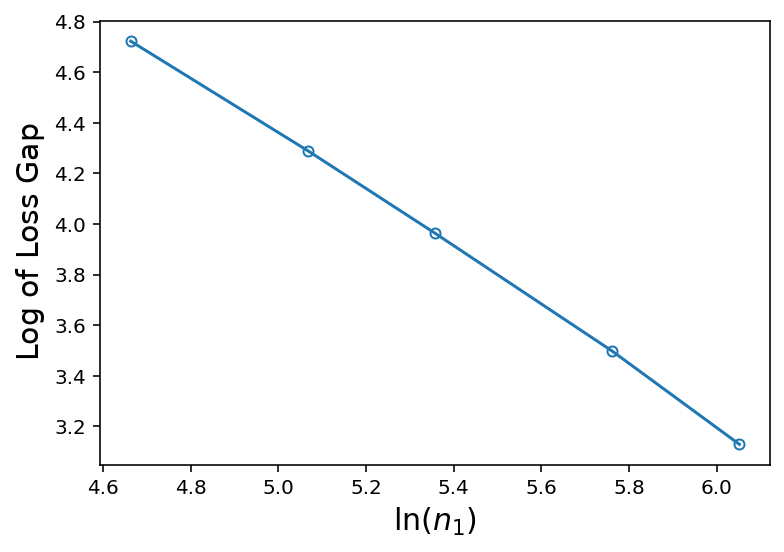}
\vspace{-5pt}
\caption{Excess Loss versus $n_1$ on 
Digit (top-left), MNIST (top-mid) and 
Super (top-right), and corresponding Loss 
Gap on Digit (bottom-left), MNIST 
(bottom-mid) and Super (bottom-right).}
\label{fig:error}
\end{figure}

\begin{figure}
\centering
\includegraphics[width=.32\linewidth]{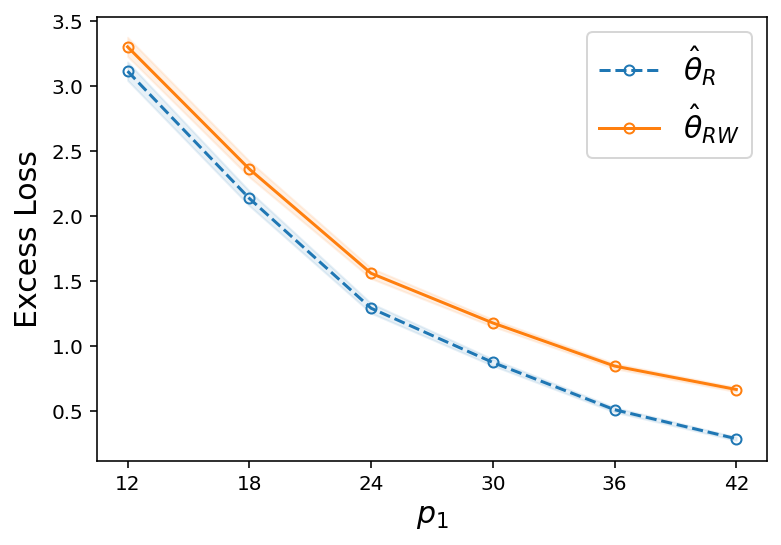}
\includegraphics[width=.32\linewidth]{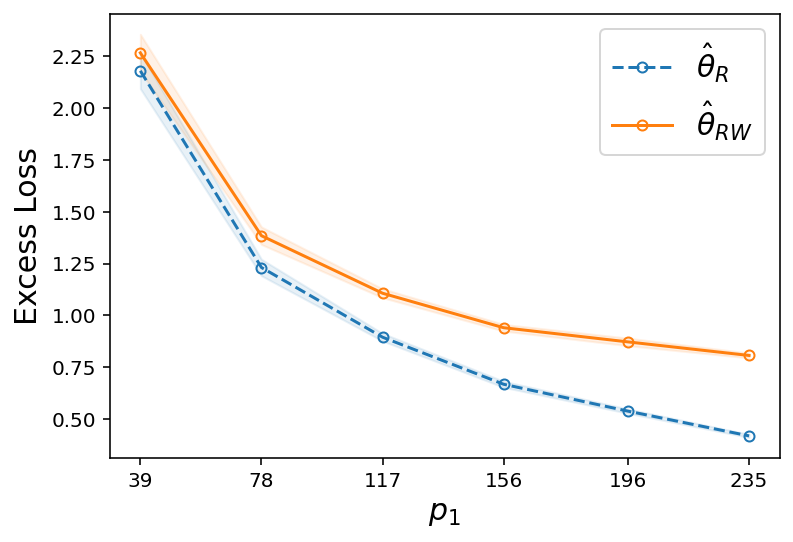}
\includegraphics[width=.32\linewidth]{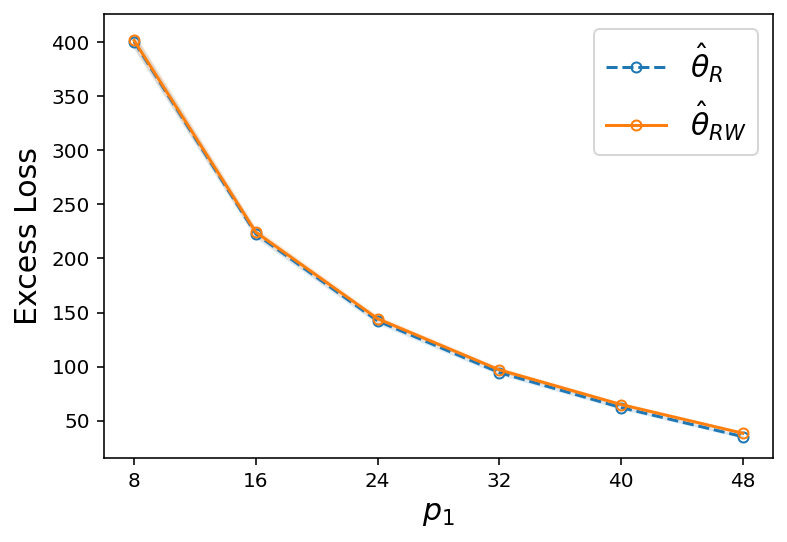}\\
\includegraphics[width=.32\linewidth]{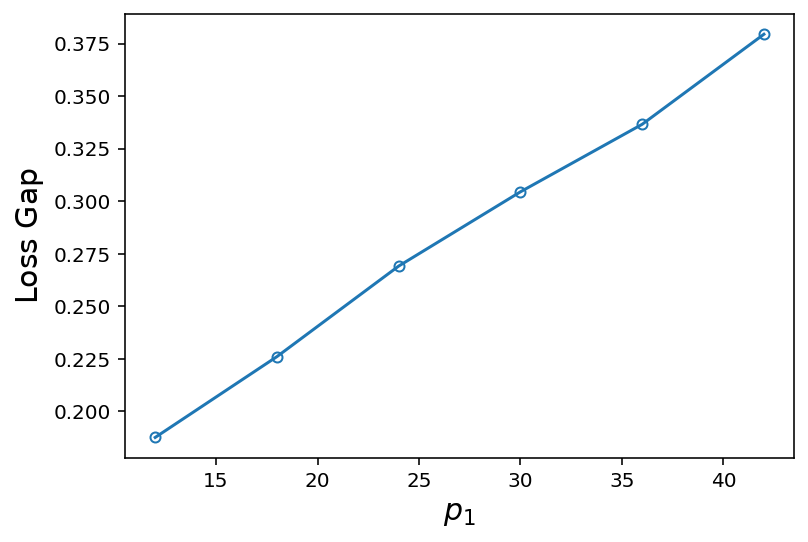}
\includegraphics[width=.32\linewidth]{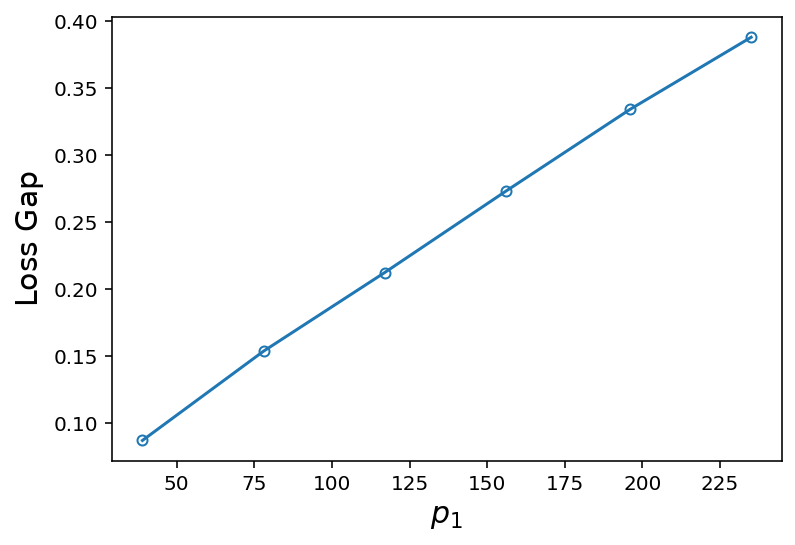}
\includegraphics[width=.32\linewidth]{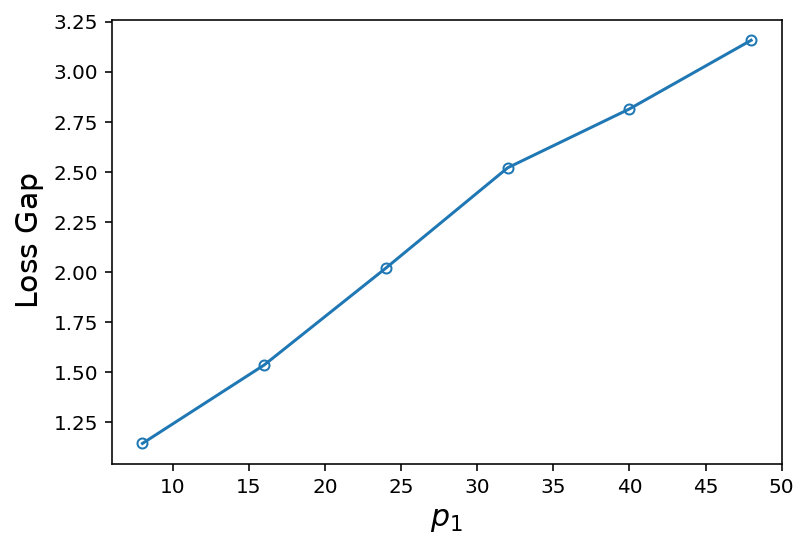}
\vspace{-5pt}
\caption{Excess Loss versus $p_1$ on 
Digit (top-left), MNIST (top-mid) and 
Super (top-right), and corresponding Loss 
Gap on Digit (bottom-left), MNIST 
(bottom-mid) and Super (bottom-right).}
\label{fig:error_p}
\end{figure}

\section{Acknowledgments}

We thank the anonymous reviewer 
for the feedback that improved the 
quality of the work. 

Theorem \ref{cor:brpbound} was 
developed by the authors, 
but the original choice of 
sketching constant made the 
bound vacuous. To fix this, 
Claude suggested an alternative constant, 
which is $\la_0$ in (\ref{eq:lambda0}). 
This changes Claim \ref{claim:05}, 
but not the structure of the theorem 
or its proof. 
Analysis of $\la_0$ was largely 
suggested and drafted by Claude, 
and then digested, verified, simplified 
and written by the authors. Analysis for 
Implication 3 was refined by Claude.

\bibliographystyle{plain}
\bibliography{reference}

\clearpage 
\appendix

\section{Proofs of Supporting Claims for Theorem 
\ref{cor:brpbound}}

Recall the random noise vector 
$E \in \R^n$ satisfies $\E_{Y \mid R, W} [E] 
= \vec{0}$ and $\E_{Y \mid R, W} 
[EE^T] = \sigma^2 I_{n}$. 

\begin{claim}
\label{claim:02}
$\E_{Y \mid R,W} 
[\hte_{RW}] = \hte_*$. 
\end{claim}
\begin{proof}
Let $\tX = WXR$ and $\tY = WY$. 
Fix $(R, W)$ and note that $\tX^T \tX$ 
is invertible almost surely. 

Given $Y$, since $\hte_{RW}$ minimizes 
$\frac{1}{n_1}||\tX \te - \tY||^2$, 
\begin{equation}
\label{eq:claim02_03}
\hte_{RW} = (\tX^T \tX)^{-1} \tX^T \tY,    
\end{equation}
and hence 
\begin{equation}
\label{claim2:01}
\E_{Y \mid R, W} [\hte_{RW}] 
= (\tX^T \tX)^{-1} \tX^T \E[\tY].
\end{equation}

On the other hand, since $\hte_*$ 
minimizes the following loss 
\begin{equation}
L_{RW}(\te) 
= \frac{1}{n_1} 
\E_{Y \mid R, W} 
||\tX \te - \tY||^2
= \frac{1}{n_1} 
\left( ||\tX \te||^2 + 
\E ||\tY||^2 - 2 \E[\tY]^T(\tX \te)\right).
\end{equation}
By taking derivative of the right 
side w.r.t. $\te$ and setting it 
to zero, we have 
\begin{equation}
\label{claim2:02}
\hte_* = (\tX^T \tX)^{-1} 
\tX^T \E_{Y \mid R, W} [\tY].  
\end{equation}
Combining (\ref{claim2:01}) 
and (\ref{claim2:02}) proves the claim. 
\end{proof}

\begin{claim}
\label{claim:03}
$\E_{\hte_{RW} \mid R,W} 
[L_{RW}(\hte_{RW})] \leq 
\E \frac{1}{n_1} || W XR 
\hte_{RW} - W XR \hte_* ||^2 
+ L_{RW}(\hte_*)$.
\end{claim}
\begin{proof}
For convenience, 
let $\tX = WXR$ and $\tY = WY$. 
Fix $(R,W)$ in all following discussions. 

First, we observe that 
\begin{equation}
\label{eq:claim03:prf_01}
\begin{split}
& n_1 \cdot \E_{\hte_{RW} \mid R,W} 
[L_{RW}(\hte_{RW})] \\
& = \E_{\hte_{RW}, Y \mid R, W} 
||\tX \hte_{RW} - \tY||^2 \\ 
& = \E ||\tX \hte_{RW} - \tX \hte_{*} 
+ \tX \hte_{*} - \tY||^2 \\
& = \E ||\tX \hte_{RW} - \tX \hte_{*} 
||^2 + \E || \tX \hte_{*} - \tY||^2 
+ 2 \E (\tX \hte_{RW} - \tX \hte_{*} )^T 
(\tX \hte_{*} - \tY). 
\end{split}    
\end{equation}
The last term is non-positive due 
to the following arguments: 
\begin{equation}
\label{eq:arg_09}
\begin{split}
& \E_{\hte_{RW}, Y \mid R, W} 
(\tX \hte_{RW} - \tX \hte_{*} )^T 
(\tX \hte_{*} - \tY)    \\
& = \E ( \hte_{RW}^T \tX^T \tX \hte_* 
- \hte_{RW}^T \tX^T \tY
- \hte_{*}^T \tX^T \tX \hte_{*} 
+ \hte_{*}^T \tX^T \tY)\\
& =  \E[\hte_{RW}]^T \tX^T \tX \hte_* 
- \E[ \hte_{RW}^T \tX^T \tY]
- \hte_{*}^T \tX^T \tX \hte_{*} 
+ \hte_{*}^T \tX^T \E[\tY]\\
& = - \E[ \hte_{RW}^T \tX^T \tY] 
+ \hte_{*}^T \tX^T \E[\tY]\\
& = - \E[ \tY^T \tX (\tX^T \tX)^{-1} 
\tX^T \tY] + \E[\tY]^T \tX (\tX^T 
\tX)^{-1} \tX^T \E[\tY]\\ 
& = - \E[ E^T W^T \tX (\tX^T \tX)^{-1} 
\tX^T W E]\\
& \leq 0.
\end{split}
\end{equation}
In (\ref{eq:arg_09}), the third line 
follows from $\E_{\hte_{RW} \mid R, W}[\hte_{RW}] 
= \hte_*$ (Claim \ref{claim:02}), 
the fifth line follows from 
$\tY = W X \be_* + WE$ 
(by the fixed design setting), 
and the last inequality holds 
since $W^T \tX (\tX^T \tX)^{-1} 
\tX^T W$ is positive 
semi-definite. Plugging 
this into (\ref{eq:claim03:prf_01}) 
and dividing both sides by $n_1$
proves the claim. 
\end{proof}

\begin{claim}
\label{claim:04}
$\E_{\hte_{RW} \mid R,W} 
\frac{1}{n_1} 
|| WXR \hte_{RW} - WXR \hte_* ||^2 
\leq \sigma^2 \frac{p_1}{n_1} 
\la_{\uparrow}^2(W)$.
\end{claim}
\begin{proof}
For convenience, 
let $\tX = WXR$ and $\tY = WY$. 
Fix $(R,W)$ in all following discussions. 

By the fixed design setting 
and that $\E_Y [E] = 0$, 
\begin{equation}
\tY - \E_{Y \mid R, W}[\tY] 
= (W X \be_* + W E)  
- (W X \be_* + W\, \E[E]) 
= WE.
\end{equation}
Combining this with 
(\ref{eq:claim02_03}, 
\ref{claim2:02}), we have that 
\begin{equation}
\hte_{RW} - \hte_{*} 
= (\tX^T \tX)^{-1} \tX^T 
(\tY - \E_{Y \mid R, W} [\tY]) 
= (\tX^T \tX)^{-1} \tX^T W E.
\end{equation}
Based on this, 
\begin{equation}
\label{eq:arg_10}
\begin{split}
\frac{1}{n_1} \E_{\hte_{RW} \mid R,W}  
|| \tX \hte_{RW} - \tX \hte_* ||^2 
& = \frac{1}{n_1} \E\ ||\tX (\tX^T \tX)^{-1} \tX^T W E||^2\\ 
& = \frac{1}{n_1} \E\ \mtr ( 
E^T W^T \tX (\tX^T \tX)^{-1} \tX^T
\ \tX (\tX^T \tX)^{-1} \tX^T W E )\\
& = \frac{1}{n_1} 
\E\ \mtr ( \tX (\tX^T \tX)^{-1} \tX^T
\tX (\tX^T \tX)^{-1} \tX^T\ W E E^T W^T )\\
& = \frac{1}{n_1} 
\mtr ( \tX (\tX^T \tX)^{-1} \tX^T
\tX (\tX^T \tX)^{-1} \tX^T W\ 
\E [E E^T]\ W^T )\\
& = \frac{\sigma^2}{n_1} 
\mtr ( \tX (\tX^T \tX)^{-1} \tX^T
\tX (\tX^T \tX)^{-1} \tX^T W W^T )\\
& \leq \frac{\sigma^2}{n_1} 
\lambda^2_{\uparrow}(W)\ 
\mtr ( \tX (\tX^T \tX)^{-1} \tX^T
\tX (\tX^T \tX)^{-1} \tX^T )\\
& = \frac{\sigma^2}{n_1} 
\lambda^2_{\uparrow}(W)\ 
\mtr ( I_{p_1} ) \\
& = \sigma^2 \frac{p_1}{n_1} 
\lambda^2_{\uparrow}(W).
\end{split}
\end{equation}
In (\ref{eq:arg_10}), the fourth line is by 
the linearity of expectation,  
the fifth line follows from the fact 
that $\E_{Y \mid R, W} [EE^T] 
= \sigma^2 I_{n}$, and the 
sixth line follows from a basic 
trace inequality $\mtr(B A) \leq 
\la_{\uparrow}(A) \mtr(B)$. 
This proves the claim. 
\end{proof}

\begin{remark}
\label{rem:priorkaban}
$\E_{R} [L_{R}(R^T\be_*)] 
\leq L(\be_*) + \frac{1}{p_1} 
||\be_*||_M^2$, 
where $M = \Sigma + tr(\Sigma) I_p$ and 
$\Sigma = \frac{1}{n} X^T X$.     
\end{remark}
\begin{proof}
This is an intermediate result 
in the proof of \citep[Theorem 1]{kaban2014new}. 
In the reference paper, 
it was shown through their 
arguments (15-17) that 
\begin{equation}
L_R(R^T \be_*) 
= L(\be_*) + 
\frac{1}{n} ||X \be_* 
- X R R^T \be_*||^2    
\end{equation}
and through their arguments (19-24) that 
\begin{equation}
\frac{1}{n} \E_{R} \left[ ||X \be_* 
- X R R^T \be_*||^2 \right]
\leq \frac{1}{p_1} ||\be_*||^2_{M}.    
\end{equation}
 Combining both yields 
the remark. (It is noted that, 
in the original proof, $w$ represents 
the true model which is represented 
as $\be_*$ in our study; also, projection 
$R$ is transposed in our study.) 
\end{proof}

\begin{claim}
\label{claim:05}
$\mathbb{E}_{R,W,\hat\theta_{RW}}[L_{RW}(\hat\theta_{RW})] \;\ge\; \frac{n}{n_1}\,\lambda_0\;\mathbb{E}_{R,\hat\theta_{RW}}[L_R(\hat\theta_{RW})]$.
\end{claim}

\begin{proof}
Recall that $\Set$ is the column space of $[XR, Y]$, 
and $\la_{\downarrow \Set}(W)\ 
:= \min_{z\in\mathcal{S},\, z\neq 0} 
\frac{\|Wz\|}{\|z\|}$.

By definition, the following lower 
bound holds for any $z \in \Set$: 
\begin{equation}\label{eq:restricted}
\la_{\downarrow \Set}^2(W)\,\|z\|^2 \;\le\; \|Wz\|^2.
\end{equation}
Based on this and the fact that 
$XR\hat\theta_{RW}-Y\in\mathcal{S}$ 
for every realization of $(R,W,E)$,  
\begin{equation}
\label{eq:newprf_01}
\begin{split}
\E_{R,W,\hte_{RW}}[L_{RW}(\hte_{RW})]
&  = \frac{1}{n_1}\E 
|| W(XR\hte_{RW}-Y)||^2\\ 
& \ge \frac{1}{n_1}\E [ 
\la_{\downarrow \Set}^2(W)\, 
|| XR\hte_{RW}-Y||^2 ]\\
& =  \frac{1}{n_1}\,\lambda_0 
\cdot \E || XR\hte_{RW}-Y||^2\\
& = \frac{n}{n_1}\, \lambda_0 
\cdot \E [L_R(\hat\theta_{RW})], 
\end{split}
\end{equation}
where the third line follows from 
definition $\la_0 = \frac{\E_{R,W,Y} [ 
\lambda_{\downarrow \Set}^2(W)\, 
\|XR\hat\theta_{RW}-Y\|^2 ]}
{\E_{R,W,Y} \|XR\hat\theta_{RW} -Y\|^2}$, 
and the last line follows from 
the definition of $L_R(\hte_{RW})$. 
This proves the claim. 
\end{proof}

\begin{claim}
\label{claim:05b}
$L(\be_*) = \sigma^2$.
\end{claim}
\begin{proof}
The claim follows from the definition, i.e., 
\begin{equation}
\begin{split}
L(\be_*) 
= \frac{1}{n} \E \| X \be_* - Y\|^2
= \frac{1}{n} \E \| E\|^2
= \frac{1}{n} \mtr[\sigma^2 I_n]  
= \sigma^2.
\end{split}
\end{equation}
\end{proof}

\section{Proofs of Lemmas, 
Remarks and other Claims}

\begin{lemma*}[\textbf{\ref{lemma:baseU}}]
Let $U \in \R^{n \times k_s}$ have 
orthonormal columns that form a basis 
for $\Set$. Then, given $(R,Y)$, 
it follows that (i) $\la_{\downarrow}(WU) 
= \la_{\downarrow \Set}(W)$ 
and (ii) $\sqrt{n} WU$ has 
i.i.d. $N(0,1)$ entries. 
\end{lemma*}
\begin{proof}
Note that $U$ is fixed 
given $(R, Y)$. Moreover, 
for every $\theta \in \mathbb{R}^{p_1}$,  
$XR\theta - Y \in \mathcal{S}$ 
because the difference is a linear 
combination of $[XR,Y]$ columns. 

We first justify $\la_{\downarrow}(WU) 
= \la_{\downarrow \Set}(W)$. It is true 
because every $z \in \Set$ can be written 
as $U c$ for some $c \in \R^{k_s}$ (due 
to the fact that $U$ is a basis for 
$\Set$). Consequently, 
\begin{equation}
\lambda_{\downarrow \Set}(W)\ 
= \min_{z\in\mathcal{S},\, z\neq 0} \frac{\|Wz\|}{\|z\|} 
= \min_{c \in \R^{k_s},\, c\neq 0} 
\frac{\|WU c\|}{\|U c\|}
= \min_{c \in \R^{k_s},\, c\neq 0} 
\frac{\|WU c\|}{\|c\|}
= \lambda_{\downarrow}(WU),  
\end{equation}
where the third equality holds 
because $\| Uc\|^2 = c^T U^T U c 
= c^T I_{k_s} c = \| c\|^2$.

Now we justify that 
$\sqrt{n}WU$ has i.i.d. $N(0,1)$ 
entries. 

First, the matrix has i.i.d. rows 
simply because $W$ has i.i.d. rows. 

Second, in each row, 
entries are independent because 
they are jointly Gaussian and 
uncorrelated. 
To be specific, let $w$ denote an arbitrary 
row of $W$. Entries of $wU$ are jointly 
Gaussian because their linear combination 
can be written as $w U c$ for some 
$c \in \R^{k_s}$, 
which is a linear combination of 
the Gaussian entries in $w$ and 
thus is itself Gaussian. Moreover, 
entries of $wU$ are uncorrelated 
because $\E [wU] = \vec{0}$ and thus 
their covariance matrix 
is $\E (wU)^T (wU) = U^T \E[w^T w] U 
= \frac{1}{n} U^T U = \frac{1}{n} I_{k_s}$. 

Finally, in each row, 
entries follow identical distribution 
because each of them is a linear combination 
of the i.i.d. $w$ entries with a unit-norm 
linear coefficient vector -- as 
a result, each entry inherits the 
$N(0,1/n)$ distribution of $w$ entry. 
Scaling this by $\sqrt{n}$ yields $N(0,1)$. 
\end{proof}

\vspace{5pt}
\begin{lemma*}[\textbf{\ref{claim:SVW_concentration}}]
Given $(R,Y)$, for any $t > 0$, 
$|\, \la_{\downarrow \Set}(W) 
- \E_{W \mid R, Y} 
[\la_{\downarrow \Set}(W)]\, 
|  \leq t  / \sqrt{n}$ 
with probability at least 
$1 - 2 e^{-t^2/2}$ over 
the random choice of $W$. 
\end{lemma*}
\begin{proof}
Lemma \ref{lemma:baseU} 
states that for an orthonormal 
basis $U$ for $\Set$, we have 
$\la_{\downarrow}(WU) = 
\la_{\downarrow \Set}(W)$ 
and $\sqrt{n} WU$ has i.i.d 
$N(0,1)$ entries.
In the following, we 
show that $\la_{\downarrow}(WU)$ 
concentrates tightly around its 
mean, because $\sqrt{n}WU$ has 
i.i.d. standard Gaussian entries. 

The following result is from 
\cite[Proposition 5.34]{vershynin2012introduction}, 
which states that Lipschitz function 
(w.r.t. the Frobenius norm) 
preserves the tight concentration 
of a standard Gaussian vector. 

\begin{proposition}[\cite{vershynin2012introduction}]
\label{prop:cvconcentration}
Let $f$ be a real valued function 
on $\R^n$ with Lipschitz constant 
$K$, i.e. $|f(x)-f(y)| \leq K 
\| x-y \|_2$ for all $x,y \in \R^n$ 
(such functions are also called 
$K$-Lipschitz). Let $X_g$ be the 
standard normal random vector 
in $\R^n$. Then for every 
$t \geq 0$ one has 
\begin{equation}
\Pr \{ f(X_g) - \E f(X_g) > t\} 
\leq \exp\left(-\frac{t^2}{2K^2}\right).
\end{equation}
\end{proposition}

The next result is from 
\cite[Corollary 7.3.5]{horn2012matrix}.  
It shows that the singular value 
function is Lipschitz w.r.t. 
the induced $\ell_2$ norm 
and thus w.r.t. the Frobenius 
norm (since $\|A - B \|_{op} 
\leq \| A - B \|_F$). It also 
suggests that the Lipschitz 
constant is at most one. 

\begin{proposition}[\cite{horn2012matrix}]
\label{prop:svbound}
Let $A,B \in \R^{n \times m}$ 
and let $q = \min\{m, n\}$. Let 
$\sigma_1(A) \geq \ldots \geq \sigma_q(A)$
and 
$\sigma_1(B) \geq \ldots \geq \sigma_q(B)$ 
be the nonincreasingly ordered singular 
values of $A$ and $B$, respectively. Then
\begin{equation}
|\sigma_i(A) - \sigma_i(B)| 
\leq \|A - B \|_{op},
\end{equation}
for each $i = 1, \ldots, q$, 
where $\|\cdot\|_{op}$ is an 
induced $\ell_2$ norm. 
\end{proposition}

Now we apply Proposition 
\ref{prop:cvconcentration} 
on the least singular value of 
$\sqrt{n} WU$. Three remarks are needed. 

First, singular 
value is a function of matrices 
in $\R^{n \times m}$, but is 
equivalent to a function of vectors 
in $\R^{nm}$ that remains Lipschitz 
w.r.t. the $\ell_2$ vector norm 
and has the same Lipschitz constant.

Second, Lipschitz constant $K = 1$ 
according to Proposition \ref{prop:svbound}. 

Third, Proposition \ref{prop:cvconcentration} 
is for one side, but covers the other side 
because if $f$ is Lipschitz 
then $-f$ remains so. Combining 
both through a union bound yields 
a $2 \exp(-t^2/2K^2)$ failure probability. 

With the above remarks, applying \ref{prop:cvconcentration} 
on the least singular value of 
$\sqrt{n} WU$ yields 
\begin{equation}
\Pr_{W \mid R,Y} \{ | \la_{\downarrow}(
\sqrt{n} WU) - \E_{W \mid R, Y} [\la_{\downarrow}( 
\sqrt{n} WU)] | >  t \} \leq 2\exp(-t^2/2), 
\end{equation}
which implies 
\begin{equation}
\Pr_{W \mid R,Y} \{ | \la_{\downarrow}(WU) 
- \E_{W \mid R, Y} [\la_{\downarrow}(WU)] 
| >  t/\sqrt{n} \} \leq 2\exp(-t^2/2).
\end{equation}
Combining it with $\la_{\downarrow}(WU) 
= \la_{\downarrow \Set}(W)$ proves the lemma. 
\end{proof}

\vspace{5pt}
\begin{lemma*}[\textbf{\ref{claim:SVW_mean}}]
Given $(R,Y)$, 
$\frac{\sqrt{n_1} - \sqrt{p_1+1}}{
\sqrt{n}}\ \leq\  \E_{W \mid R, Y} 
[\la_{\downarrow \Set}(W)]\ 
\leq\ \sqrt{\frac{n_1}{n}}$.
\end{lemma*}
\begin{proof}
We first prove the upper bound. 
Let $G = \sqrt{n} WU$. Lemma 
\ref{lemma:baseU} suggests that 
\begin{equation}
\label{eq:prf_001}
\la_{\downarrow \Set}(W) 
=  \la_{\downarrow}(W U) 
= \frac{\la_{\downarrow}(G)}{\sqrt{n}}.
\end{equation}
Let $e_1 = [1, 0, \ldots, 0]^T \in \R^{k_s}$ 
and $G_{:1} \in \R^{n_1}$ be the first 
column of $G$. By definition, 
\begin{equation}
\label{eq:claim_ElaBounds_01}
\la_{\downarrow}(G) = \min_{z \neq 0} 
\frac{\| G z\|}{\| z\|}
\leq \frac{\| G e_1\|}{\| e_1\|} 
= \| G_{:1} \|.
\end{equation}
Since $G$ has i.i.d. $N(0,1)$ 
entries, $\E [G_{:1} ^T G_{:1}] = n_1$ 
and by Jensen's inequality, 
\begin{equation}
\E_W \| G_{:1} \| 
= \E \sqrt{G_{:1} ^T G_{:1}} 
\leq \sqrt{\E [G_{:1} ^T G_{:1}]}
= \sqrt{n_1}.
\end{equation}
Taking expectation w.r.t. 
$W$ on both sides of 
(\ref{eq:claim_ElaBounds_01}) 
yields the upper bound, i.e.,  
\begin{equation}
\E_{W \mid R,Y} [\la_{\downarrow \Set}(W)] 
= \frac{\E \la_{\downarrow}(G)}{\sqrt{n}}
\leq \frac{\E \| G_{:1} \|}{\sqrt{n}}
\leq \frac{\sqrt{n_1}}{\sqrt{n}}.
\end{equation}

The lower bound follows directly 
from the Gordon's theorem for 
matrices with independent standard 
normal entries \cite[Theorem 5.32]{vershynin2012introduction}. 
Since $G$ has such entries, the 
theorem says 
\begin{equation}
\E_{W \mid R,Y} [\la_{\downarrow}(G)] 
\geq \sqrt{n_1} - \sqrt{k_s} 
\geq \sqrt{n_1} - \sqrt{p_1+1}. 
\end{equation}
Combining this with 
(\ref{eq:prf_001}) yields 
the lower bound. Combining 
both bounds proves the lemma. 
\end{proof}

\vspace{5pt}
\begin{remark*}[\textbf{\ref{rem:la0_estimate}}]
As $n_1 \to \infty$ 
and $p_1/n_1 \to \gamma$ for some 
constant $\gamma \in [0, 1)$, 
\begin{equation}
\label{eq:la0_approx}
\la_0 \approx c\ \in \left[ 
\frac{n_1}{n}\left( 1 - 
\sqrt{\gamma}\right)^2,\ 
\frac{n_1}{n}\right]. 
\end{equation}
In particular, if $p_1$ is fixed, 
then $\la_0 \approx n_1/n$. 
\end{remark*}
\begin{proof}

We first show that $\la_{\downarrow \Set}(W)$ 
concentrates around $\E [\la_{\downarrow \Set}(W)]$
tightly in the asymptotic regime.
Lemma \ref{claim:SVW_concentration}
shows that  $\la_{\downarrow \Set}(W)$ 
concentrates around its mean with an 
$O(1/\sqrt{n})$ deviation, and 
Lemma \ref{claim:SVW_mean} suggests 
that the mean is at least 
$\frac{\sqrt{n_1} - \sqrt{p_1+1}}{\sqrt{n}}$. 
Then, the relative deviation satisfies 
\begin{equation}
\label{eq:reldif_converge}
\frac{|\, \la_{\downarrow \Set}(W) 
- \E[\la_{\downarrow \Set}(W)]\, 
|}{|\, \E[\la_{\downarrow \Set}(W)]\, |} 
\leq \frac{O( \sqrt{1/n})}{(\sqrt{n_1}-\sqrt{p_1+1})/\sqrt{n}} 
= \frac{O(\sqrt{1/n_1})}{
1 - \sqrt{(p_1+1)/n_1}} \to 0,
\end{equation}
where the equality is obtained by 
multiplying both numerator and 
denominator with $\sqrt{n/n_1}$, 
and the convergence 
follows from the premise that 
$n_1 \to \infty$ and $p_1/n_1$ 
converges to a constant.

We now show 
that $\la^2_{\downarrow \Set}(W)$ 
concentrates around $\E^2 
[\la_{\downarrow \Set}(W)]$ tightly. 
Their relative deviation is 
\begin{equation}
\begin{split}
\frac{|\la^2_{\downarrow \Set}(W) 
- \E^2 [\la_{\downarrow \Set}(W)]|}{
|\E^2 [\la_{\downarrow \Set}(W)]|} 
& = 
\frac{|\la_{\downarrow \Set}(W) 
- \E [\la_{\downarrow \Set}(W)]| 
\cdot |\la_{\downarrow \Set}(W) 
+ \E [\la_{\downarrow \Set}(W)]|}{
|\E [\la_{\downarrow \Set}(W)]| 
\cdot |\E [\la_{\downarrow \Set}(W)]|} \\
& \approx 
\frac{|\la_{\downarrow \Set}(W) 
- \E [\la_{\downarrow \Set}(W)]| 
\cdot 2 \E [\la_{\downarrow \Set}(W)]}{
|\E [\la_{\downarrow \Set}(W)]| 
\cdot |\E [\la_{\downarrow \Set}(W)]|}\\
& = \frac{2  |\la_{\downarrow \Set}(W) 
- \E [\la_{\downarrow \Set}(W)]|}{
|\E [\la_{\downarrow \Set}(W)]|} \\
& \to 0,  
\end{split}
\end{equation}
where the approximation follows from 
the above conclusion that 
$\la_{\downarrow \Set}(W) $
concentrates around its mean tightly, 
and the last line of convergence follows 
directly from (\ref{eq:reldif_converge}). 

The above suggests that, in the 
asymptotic regime, we can approximate 
$\la^2_{\downarrow \Set}(W)$ 
by $\E^2 [\la_{\downarrow \Set}(W)]$. 
Plugging this into the $\la_0$ definition 
(\ref{eq:lambda0}), 
\begin{equation}
\la_0 \approx \E^2 [\la_{\downarrow \Set}(W)].   
\end{equation}

It remains to evaluate 
$\E [\la_{\downarrow \Set}(W)]$ 
in the asymptotic regime. 
Lemma \ref{claim:SVW_mean} suggests 
that it has an upper bound of $\sqrt{n_1/n}$ 
and a lower bound 
\begin{equation}
\E [\la_{\downarrow \Set}(W)] 
\geq \sqrt{\frac{n_1}{n}}
- \sqrt{\frac{p_1+1}{n}}
= \sqrt{\frac{n_1}{n}} 
\left( 1 - \sqrt{\frac{p_1+1}{n_1}} 
\right) \to \sqrt{\frac{n_1}{n}} 
\left( 1 - \sqrt{\gamma}\right).
\end{equation}
Since the right side is nonnegative, 
squaring it completes the proof.
\end{proof}

\vspace{5pt}
\begin{claim*}[\textbf{\ref{claim:Ob2}}]
If $\E_{R,W,Y}\| 
XR \hte_{RW} - Y \|^2 \in (0, \infty)$, 
then $C_1 \ge 1$ and $C_2 \geq 1$.
\end{claim*}
\begin{proof}
Recall $C_1 = \E_W[\lambda^2_{\uparrow}(W)]\, 
\lambda^{-1}_{0}$. We will prove that 
$\E_W[\la_{\uparrow}^2(W)] \ge 1$ and 
$\lambda_0 \le n_1/n$. 

For $\E_W[\la_{\uparrow}^2(W)]$, the 
main idea is that $\la_{\uparrow}(W)$ 
is the top singular value of $W$ and 
thus its square is 
no smaller than  $\frac{1}{n_1} \mtr [WW^T]$ 
(which is easier to evaluate). To be specific, 
\begin{equation}
\la_{\uparrow}^2(W) = 
\lambda_{\uparrow}(WW^T) 
\geq \frac{1}{n_1} 
\sum_{i=1}^{n_1} 
\lambda_{i}(WW^T) 
= \frac{1}{n_1} \mtr[WW^T], 
\end{equation}
where $\la_i(\cdot)$ denotes 
the $i_{th}$ singular value. 
Taking expectation w.r.t. 
$W$ on both sides, 
\begin{equation}
\label{eq:claim33_01a}
\E_{W}[\la_{\uparrow}^2(W)] 
\geq \frac{1}{n_1} \E \mtr[WW^T]
= \frac{1}{n_1} \mtr [\E WW^T ]
= \frac{1}{n_1} \mtr [I_{n_1}]
= 1,  
\end{equation}
where $\E WW^T = I_{n_1}$ for two reasons: 
it is diagonal because $W$ has independent rows; 
its diagonal entry has expectation 1, since it 
is a sum of n squared entries, each with second 
moment $1/n$. This verifies that 
$\E_{W}[\la_{\uparrow}^2(W)] \geq 1$. 

To verify $\lambda_0 \le n_1/n$, we start 
with $\|WXR\hte_{RW} - WY\|^2$, 
decouple the impact of $W$ by replacing 
$\hte_{RW}$ with $\hte_{R}$ through some 
residual arguments, integrate this impact 
out through expectation w.r.t. $W$, 
and finally replace $\hte_{R}$ back with $\hte_{RW}$ through residual arguments. 
Details are elaborated below. 

Given $(R, W, Y)$, since $\hat\theta_{RW}$ 
minimizes $\|WXR\theta - WY\|^2$, 
\begin{equation}
\|WXR\hte_{RW} - WY\|^2 \leq 
\|W XR\hte_R - W Y\|^2 
= (XR\hte_R - Y)^T W^T W (XR\hte_R - Y).
\end{equation}
Taking conditional expectation w.r.t. 
$W$ on both sides: 
\begin{equation}
\label{eq:claim02_C1_03}
\begin{split}
\E_{W \mid R, Y} \|WXR\hte_{RW} - WY\|^2 
& \leq (XR\hte_R - Y)^T \E [W^T W] 
(XR\hte_R - Y) \\ 
& = \frac{n_1}{n} \| XR\hte_R - Y\|^2 \\
& \leq \frac{n_1}{n} \| XR\hte_{R\tilde{W}} 
- Y\|^2. 
\end{split}
\end{equation}
In (\ref{eq:claim02_C1_03}), 
$\E[W^T W] = \frac{n_1}{n} I_n$ 
because $W$ has i.i.d. columns and 
each diagonal entry of $W^T W$ is 
a sum of $n_1$ squared entries of 
$W$ whose variance is $1/n$;  
the last inequality holds for 
any $\hte_{R\tilde{W}}$ (depending on 
the given $R$ and any $\tilde{W}$) 
because $\hte_R$ minimizes 
$\| XR\te - Y\|^2$. 

Taking expectation w.r.t. 
$(R,Y,\tilde{W})$ on both 
sides of (\ref{eq:claim02_C1_03}) 
yields 
\begin{equation}
\E_{R,W,Y,\tilde{W}} \|WXR\hte_{RW} - WY\|^2 
\leq \frac{n_1}{n} \E_{R,\tilde{W},Y} 
\| XR\hte_{R\tilde{W}} - Y\|^2. 
\end{equation}
Removing $\tilde{W}$ in the left expectation 
(which does not depend on $\tilde{W}$), replacing notation $\tilde{W}$ with $W$ in the right expectation, and plugging in the 
definition of $\la_0$ shows that 
$\lambda_0 \le n_1/n$. 

Combining $\E_{W}[\la_{\uparrow}^2(W)] \geq 1$
and $\lambda_0 \le n_1/n$ proves 
the claim that $C_1 \geq 1$. 
Claim $C_2 \geq 1$ is an immediate result 
of $\lambda_0 \le n_1/n$.
\end{proof}

\end{document}